\documentclass[a4paper,12pt]{article}

\usepackage{amsfonts}
\usepackage{amsmath}
\usepackage{amssymb}
\usepackage{amsthm}
\usepackage{cite}
\usepackage{enumitem}
\usepackage[pdfstartview=FitV,bookmarks=false]{hyperref}
\usepackage[top=20mm,bottom=30mm,left=20mm,right=20mm]{geometry}
\usepackage[notref,notcite,final]{showkeys}

\allowdisplaybreaks

\newtheorem{theorem}{Theorem}
\newtheorem{lemma}{Lemma}
\newtheorem{corollary}{Corollary}

\newtheorem*{philosinequality}{Philos' Inequality}
\newtheorem*{dphilosinequality}{Discrete Philos' Inequality}
\newtheorem{property}{Property}

\theoremstyle{definition}

\newtheorem{definition}{Definition}
\newtheorem{example}{Example}

\newtheorem{remark}{Remark}

\newcommand{\cnt}[1]{\mathrm{C}^{#1}}
\newcommand{\crd}[1]{\mathrm{C}_{\mathrm{rd}}^{#1}}
\newcommand{\reg}[1]{\mathcal{R}^{#1}}

\newcommand{\gf}[1]{\mathrm{g}_{#1}}
\newcommand{\hf}[1]{\mathrm{h}_{#1}}
\newcommand{\ef}[1]{\mathrm{e}_{#1}}

\newcommand{\N}{\mathbb{N}}
\newcommand{\R}{\mathbb{R}}
\newcommand{\T}{\mathbb{T}}
\newcommand{\Z}{\mathbb{Z}}

\title{Philos' inequality on time scales and its application\\ in the oscillation theory}

\author{Ba\c{s}ak Karpuz}

\begin{document}

\date{\footnotesize Department of Mathematics, Faculty of Science, Dokuz Eylul University, Izmir 35160, Turkey}
\maketitle

\begin{abstract}
In [Bull.\ Acad.\ Polon.\ Sci.\ S\'{e}r.\ Sci.\ Math.\ 29 (1981), no.~7-8, 367{\textendash}370], Philos proved the following result:
Let $f:[t_{0},\infty)_{\R}\to\R$ be an $n$-times differentiable function such that $f^{(n)}(t)\leq0$ ($\not\equiv0$) and $f(t)>0$ for all $t\geq{}t_{0}$.
If $f$ is unbounded, then $f(t)\geq\frac{\lambda{}t^{n-1}}{(n-1)!}f^{(n-1)}(t)$ for all sufficiently large $t$,
where $\lambda\in(0,1)_{\R}$.
In this work, we first present time scales unification of this result.
Then, by using it, we provide sufficient conditions for oscillation and asymptotic behaviour of solutions to higher-order neutral dynamic equations.\smallskip\\
\noindent\textbf{Keywords}: Asymptotic behaviour, dynamic equations, higher-order, oscillation, time scales.\\
\noindent\textbf{Mathematics Subject Classification 2010}: 34K11 (Primary), 34K40 (Secondary).
\end{abstract}

\section{Introduction}\label{intro}


In this paper, we will study oscillation of solutions to the higher-order delay dynamic equations of the form
\begin{equation}
\bigl[x(t)+A(t)x\bigl(\alpha(t)\bigr)\bigr]^{\Delta^{n}}+B(t)x\bigl(\beta(t)\bigr)=0
\quad\text{for}\ t\in[t_{0},\infty)_{\T},\label{introeq1}
\end{equation}
where $n\in\N$, $\T$ is a time scale unbounded above, $t_{0}\in\T$, $A\in\crd{}([t_{0},\infty)_{\T},\R)$ and $B\in\crd{}([t_{0},\infty)_{\T},\R_{0}^{+})$, and $\alpha,\beta\in\crd{}([t_{0},\infty)_{\T},\T)$ are unbounded nondecreasing functions such that $\alpha(t),\beta(t)\leq{}t$ for all $t\in[t_{0},\infty)_{\T}$.
We will confine our attention to the following ranges of the coefficient $A$.
\begin{enumerate}[label={(R\arabic*)},leftmargin={*},align={left},ref={(R\arabic*)}]
\item\label{r1} $A\in\crd{}([t_{0},\infty)_{\T},[0,1]_{\R})$ with $\limsup_{t\to\infty}A(t)<1$.
\item\label{r2} $A\in\crd{}([t_{0},\infty)_{\T},[-1,0]_{\R})$ with $\liminf_{t\to\infty}A(t)>-1$.
\end{enumerate}
The qualitative theory of dynamic equations has been developing faster for second-order and first-order equations
when compared to higher-order equations.
Although the theory of dynamic equations unifies the theories of differential and of difference equations,
one can see that there is not much accomplished for higher-order dynamic equations.
This is caused by the technical obstacles in the computations in the proofs and the absence
of the dynamic generalizations of the basic inequalities one of which is the so-called Philos' inequality
which we will prove its time scales generalization here.

Philos' inequality reads as follows.

\begin{philosinequality}[{\protect\cite[Lemma~2]{MR0640329}}]\label{pithmc}
Assume that $n\in\N$ and $f\in\cnt{n}([t_{0},\infty),\R^{+})$ with $f^{\Delta^{n}}\leq0$ ($\not\equiv0$) on $[t_{0},\infty)$.
If $f$ is unbounded, then we have
\begin{equation}
f(t)\geq\frac{(t-s)^{n-1}}{(n-1)!}f^{(n-1)}(t)
\quad\text{for all}\ t\geq{}s,\notag
\end{equation}
where $s\geq{}t_{0}$ is sufficiently large.
\end{philosinequality}

A discrete counterpart of Philos' inequality is given in \cite{MR1155840},
which reads as follows.

\begin{dphilosinequality}[{\protect\cite[Corollary~1.8.12]{MR1155840}}]\label{pithmd}
Let $\{f(t)\}$ be a sequence defined for $t=t_{0},t_{0}+1,\cdots$, and $f(t)>0$ and $\Delta^{n}f(t)\leq0$ ($\not\equiv0$) for $t=t_{0},t_{0}+1,\cdots$.
Then, there exists a large integer $s\geq{}t_{0}$ such that
\begin{equation}
f(t)\geq\frac{(t-s)^{(n-1)}}{(n-1)!}\Delta^{n-1}f(2^{n-m-1}t)
\quad\text{for all}\ t=s,s+1,\cdots,\notag
\end{equation}
where $^{(\cdot)}$ denotes the falling factorial function
and $m$ is the key number in discrete Kiguradze's lemma (\cite[Theorem~1.8.11]{MR1155840}).
\end{dphilosinequality}

Philos' inequality and its consequences,
which have been reference for a large number of papers,
can be regarded as one of the corner stones in the oscillation theory of higher-order delay differential equations.
A result similar to this is attended to be proved in \cite[Lemma~5]{MR3029346},
however there are some inconsistencies in its proof.
We will state and prove the dynamic generalization of Philos' inequality,
which covers the one for continuous case and improves the one for discrete case.
After proving the dynamic generalization of Philos' inequality,
we will provide easily verifiable and efficient comparison tests
for the oscillation and asymptotic behaviour of solutions to
higher-order dynamic equations depending on the order
and the two ranges of the neutral coefficient given above.

Some results for the asymptotic behaviour of solutions of higher-order dynamic equations
can also be found in \cite{MR1678096,MR2596170,MR2900835,MR3019096,MR3065777,MR2557101,MR2506155,MR3091942,MR3393378,bk17,MR2831532,MR3029346}.
As we will be making comparison with first-order dynamic equations, we find useful to redirect the readers to the papers \cite{MR2536378,MR2397023,MR2145447,MR2475963,MR2683912,MR3066722,MR3518254,MR2367669},
where they can find the most important oscillation tests for first-order dynamic equations.

To give an exact definition of a solution for the delay dynamic equation \eqref{introeq1},
we need to define $t_{-1}:=\min\{\alpha(t_{0}),\beta(t_{0})\}$.


\begin{definition}[Solution]
A function $x:[t_{-1},\infty)_{\T}\to\R$,
which is rd-continuous on $[t_{-1},t_{0}]_{\T}$ and $x+A\cdot{}x\circ\alpha$ is $n$ times $\Delta$-differentiable on $[t_{0},\infty)$,
is called a solution of \eqref{introeq1} provided that it satisfies the functional delay equation \eqref{introeq1} identically on $[t_{0},\infty)$.
\end{definition}


It can be shown as in \cite{MR2793808} that \eqref{introeq1} admits a unique solution, which exists on the entire interval $[t_{-1},\infty)_{\T}$,
when an rd-continuous initial function $\varphi:[t_{-1},t_{0}]_{\T}\to\R$ is prescribed.
More precisely, we mean in the equation that $x(t)=\varphi(t)$ for $t\in[t_{-1},t_{0}]_{\T}$.


\begin{definition}[Oscillation]
A solution $x$ of \eqref{introeq1} is called nonoscillatory if there exists $s\in[t_{0},\infty)_{\T}$
such that $x$ is either positive or negative on $[s,\infty)$.
Otherwise, the solution is said to oscillate (or is called oscillatory).
\end{definition}


The outline of the paper is organized as follows.
\S~\ref{tl} contains some fundamental results on qualitative properties of functions on time scales,
and we prove Philos' inequality in its subsection \S~\ref{pi}.
In the subsection \S~\ref{rr}, we quote some recent results on the oscillation/nonoscillation of dynamic equations,
which will be required in the sequel.
\S~\ref{mr} consists of two subsections.
In the first subsection \S~\ref{nne}, we give some comparison theorems on the qualitative behaviour of higher-order
delay dynamic equations without a neutral term,
and in the second subsection \S~\ref{ne}, we extend these results
to higher-order delay dynamic equations with a neutral term.
In the appendix section \S~\ref{app}, we present a brief introduction to the time scales calculus
and supply some important results concerning the properties of the polynomials on time scales.


\section{Technical lemmas}\label{tl}


In this section, we will form the basic facilities for the proof of our main result.


\begin{lemma}[Kiguradze's lemma \protect{\cite[Theorem~5]{MR1678096}}]\label{tllm1}
Assume that $\sup\T=\infty$, $n\in\N$ and $f\in\crd{n}([t_{0},\infty)_{\T},\R_{0}^{+})$.
Suppose that either $f^{\Delta^{n}}\geq0$ ($\not\equiv0$) or $f^{\Delta^{n}}\leq0$ ($\not\equiv0$) on $[t_{0},\infty)_{\T}$.
Then, there exist $s\in[t_{0},\infty)_{\T}$ and $m\in[0,n)_{\Z}$ such that $(-1)^{n-m}f^{\Delta^{n}}(t)\geq0$ for all $t\in[s,\infty)_{\T}$.
Moreover, the following assertions hold.
\begin{enumerate}[label={(\roman*)},leftmargin={*},align={left},ref={(\roman*)}] 
\setcounter{enumi}{0}
\item\label{tllm1it1} $f^{\Delta^{k}}(t)>0$ holds for all $t\in[s,\infty)_{\T}$ and all $k\in[0,m)_{\Z}$.
\item\label{tllm1it2} $(-1)^{m+k}f^{\Delta^{k}}(t)>0$ holds for all $t\in[s,\infty)_{\T}$ and all $k\in[m,n)_{\Z}$.
\end{enumerate}
\end{lemma}


\begin{lemma}[{\protect\cite[Lemma~7]{MR1678096}}]\label{tllm2}
If $\sup\T=\infty$, $n\in\N$ and $f\in\crd{n}([t_{0},\infty)_{\T},\R)$, then the following conditions are true.
\begin{enumerate}[label={(\roman*)},leftmargin={*},align={left},ref={(\roman*)}]
\setcounter{enumi}{0}
\item\label{tllm2it1} $\liminf_{t\to\infty}f^{\Delta^{n}}(t)>0$ implies $\lim_{t\to\infty}f^{\Delta^{k}}(t)=\infty$ for all $k\in[0,n)_{\Z}$.
\item\label{tllm2it2} $\limsup_{t\to\infty}f^{\Delta^{n}}(t)<0$ implies $\lim_{t\to\infty}f^{\Delta^{k}}(t)=-\infty$ for all $k\in[0,n)_{\Z}$.
\end{enumerate}
\end{lemma}


\begin{corollary}[{\protect\cite[Corollary~2.10]{MR2900835}}]\label{tlcrl1}
If $\sup\T=\infty$, $n\in\N$ and $f\in\crd{n}([t_{0},\infty)_{\T},\R_{0}^{+})$, then
\begin{equation}
\lim_{t\to\infty}f^{\Delta^{k}}(t)=0
\quad\text{for all}\ k\in(m,n)_{\Z},\notag
\end{equation}
where $m\in[0,n)_{\Z}$ is the key number in \hyperref[tllm1]{Kiguradze's lemma}.
\end{corollary}


\subsection{Philos' inequality}\label{pi}


In this section, we present and prove the dynamic generalization of the well-known inequality \cite[Lemma~2]{MR0640329}.


\begin{theorem}[Dynamic Philos' inequality]\label{pithm1}
Assume that $\sup\T=\infty$, $n\in[2,\infty)_{\Z}$ and $f\in\crd{n}([t_{0},\infty)_{\T},\R_{0}^{+})$ with $f^{\Delta^{n}}\leq0$ ($\not\equiv0$) on $[t_{0},\infty)_{\T}$.
Then, we have
\begin{equation}
f(t)\geq\hf{n-1}(t,s)f^{\Delta^{n-1}}(t)
\quad\text{for all}\ t\in[s,\infty)_{\T},\label{pithm1eq1}
\end{equation}
where $s\in[t_{0},\infty)_{\T}$ is defined as in \hyperref[tllm1]{Kiguradze's lemma}.
\end{theorem}


To prove the dynamic generalization of \hyperref[pithm1]{Philos' inequality},
we need a series of lemmas.


\begin{remark}\label{pirmk1}
Let $\T$ be a time scale with a linear forward jump, i.e., $\sigma(t):=qt+h$ for $t\in\T$, where $q\in[1,\infty)_{\R}$ and $h\in\R_{0}^{+}$.
By induction, one can prove the following two properties.
\begin{enumerate}[label={(P\arabic*)},leftmargin={*},align={left},ref={(P\arabic*)}]
\setcounter{enumi}{0}
\item\label{pirmk1p1} $\hf{n}(t,s)=\frac{1}{\Gamma_{q}(n)}\prod_{i=0}^{n-1}\bigl(t-\sigma^{i}(s)\bigr)$ for $s,t\in\T$ and $n\in\N_{0}$,
    where $\Gamma_{q}$ is the $q$-Gamma function defined by
$\Gamma_{q}(n):=\lim_{\lambda\to{}q}\prod_{i=1}^{n-1}\frac{\lambda^{i}-1}{\lambda-1}$ for $n\in\N$.
\item\label{pirmk1p2} $\hf{n}(t,s)=(-1)^{n}q^{\frac{n(n-1)}{2}}\hf{n}\bigl(s,\rho^{n-1}(t)\bigr)$ for $s,t\in\T$ and $n\in\N_{0}$.
\end{enumerate}
It follows from \ref{pirmk1p1} and \ref{pirmk1p2} that
\begin{equation}
\lim_{t\to\infty}\frac{\hf{n}(t,s)}{t^{n}}=\frac{1}{\Gamma_{q}(n)}
\quad\text{and}\quad
\lim_{s\to\infty}\frac{\hf{n}(t,s)}{s^{n}}=(-1)^{n}\frac{q^{\frac{n(n-1)}{2}}}{\Gamma_{q}(n)}
\quad\text{for}\ s,t\in\T\ \text{and}\ n\in\N_{0}.\notag
\end{equation}
\end{remark}

\begin{remark}\label{pirmk2}
First, for the case $\T=\R$, \eqref{pithm1eq1} reads as
\begin{equation}
f(t)\geq\frac{(t-s)^{n-1}}{(n-1)!}f^{(n-1)}(t)
\quad\text{for all}\ t\in[s,\infty)_{\R}.\notag
\end{equation}
Next, for the case $\T=\Z$, \eqref{pithm1eq1} reduces to
\begin{equation}
f(t)\geq\frac{(t-s)^{(n-1)}}{(n-1)!}\Delta^{n-1}f(t)
\quad\text{for all}\ t\in[s,\infty)_{\Z},\notag
\end{equation}
where $^{(\cdot)}$ denotes the falling factorial function and $\Delta$ is the difference operator.
As $\Delta^{n-1}f$ is nonincreasing on $[s,\infty)_{\Z}$,
we have $\Delta^{n-1}f(t)\geq\Delta^{n-1}f(2^{n-m-1}t)$ for all $t\in[s,\infty)_{\Z}$.
Therefore, \hyperref[pithm1]{dynamic Philos' inequality} improves \cite[Corollary~1.8.12]{MR1155840} even in the particular case $\T=\Z$.
Finally, for the case $\T=q^{\Z}\cup\{0\}$, \eqref{pithm1eq1} becomes
\begin{equation}
f(t)\geq\prod_{i=0}^{n-1}\frac{t-q^{i}s}{\sum_{j=0}^{i}q^{j}}\mathrm{D}_{q}^{n-1}f(t)
\quad\text{for all}\ t\in[s,\infty)_{\T},\notag
\end{equation}
where $\mathrm{D}_{q}$ is the $q$-difference operator (see Table~\ref{tsetbl2} and Table~\ref{tsptbl1}).
\end{remark}


\begin{lemma}\label{pilm1}
If $k\in\N_{0}$ and $s\in\T$, then
\begin{equation}
(-1)^{k}\hf{k}(s,t)\geq\hf{k}(t,s)
\quad\text{for all}\ t\in[s,\infty)_{\T}.\notag
\end{equation}
\end{lemma}

\begin{proof}
The proof is trivial if $k=0$.
Assume that the claim is true for some $k\in\N_{0}$.
By Property~\ref{tspprp1}, we have
\begin{align}
(-1)^{k+1}\hf{k+1}(s,t)={}&(-1)^{k+1}\int_{t}^{s}\hf{k}\bigl(s,\sigma(\eta)\bigr)\Delta\eta
=(-1)^{k}\int_{s}^{t}\hf{k}\bigl(s,\sigma(\eta)\bigr)\Delta\eta\notag\\
\geq{}&\int_{s}^{t}\hf{k}\bigl(\sigma(\eta),s\bigr)\Delta\eta
\geq\int_{s}^{t}\hf{k}(\eta,s)\Delta\eta
=\hf{k+1}(t,s)\notag
\end{align}
for all $t\in[s,\infty)_{\T}$.
This shows that the inequality is also true when $k$ is replaced with $(k+1)$.
By mathematical induction, we justify the validity of the inequality for all $k\in\N_{0}$.
\end{proof}


\begin{lemma}\label{pilm2}
If $k,\ell\in\N_{0}$ and $s\in\T$, then
\begin{equation}
\hf{k}(t,s)\hf{\ell}(t,s)\geq\hf{k+\ell}(t,s)
\quad\text{for all}\ t\in[s,\infty)_{\T}.\notag
\end{equation}
\end{lemma}

\begin{proof}
The proof is obvious if $k=0$ or $\ell=0$.
Hence, we let $k,\ell\in\N$ below.
By Lemma~\ref{tsplm2}, we have
\begin{equation}
\hf{k+\ell}(t,s)=\int_{s}^{t}\hf{k-1}\bigl(t,\sigma(\eta)\bigr)\hf{\ell}(\eta,s)\Delta\eta
\quad\text{for all}\ t\in[s,\infty)_{\T}.\notag
\end{equation}
It follows from Property~\ref{tspprp1} that $\hf{\ell}(\cdot,s)$ is increasing on $[s,\infty)_{\T}$, which yields
\begin{equation}
\hf{k+\ell}(t,s)
\leq\biggl(\int_{s}^{t}\hf{k-1}\bigl(t,\sigma(\eta)\bigr)\Delta\eta\biggr)\hf{\ell}(t,s)
=\hf{k}(t,s)\hf{\ell}(t,s)
\quad\text{for all}\ t\in[s,\infty)_{\T},\notag
\end{equation}
where we have used \eqref{tspeq3} in the last step.
\end{proof}


\begin{lemma}\label{pilm3}
If $k\in\N$, $\ell\in\N_{0}$ and $s\in\T$, then
\begin{equation}
(-1)^{\ell}\int_{s}^{t}\hf{k-1}\bigl(t,\sigma(\eta)\bigr)\hf{\ell}(\eta,t)\Delta\eta\geq\hf{k+\ell}(t,s)
\quad\text{for}\ t\in[s,\infty)_{\T}.\notag
\end{equation}
\end{lemma}

\begin{proof}
The claim holds with equality for $\ell=0$ by \eqref{tspeq3}.
Below, we will consider the case where $\ell\in\N$.
Let $k,\ell\in\N$, then we have
\begin{align}
(-1)^{\ell}&\int_{s}^{t}\hf{k-1}\bigl(t,\sigma(\eta)\bigr)\hf{\ell}(\eta,t)\Delta\eta\notag\\
={}&(-1)^{\ell}\int_{s}^{t}\hf{k}\bigl(t,\sigma(\eta)\bigr)\biggl(\int_{t}^{\eta}\hf{\ell-1}(\zeta,t)\Delta\zeta\biggr)\Delta\eta\notag\\
={}&(-1)^{\ell}\int_{s}^{t}\hf{k}\bigl(t,\sigma(\eta)\bigr)\biggl(\int_{t}^{s}\hf{\ell-1}(\zeta,t)\Delta\zeta
-\int_{s}^{\eta}\hf{\ell-1}(\zeta,t)\Delta\zeta\biggr)\Delta\eta\notag\\
={}&(-1)^{\ell}\int_{s}^{t}\hf{k}\bigl(t,\sigma(\eta)\bigr)\biggl(\hf{\ell}(s,t)-\int_{s}^{\eta}\hf{\ell-1}(\zeta,t)\Delta\zeta\biggr)\Delta\eta\notag\\
={}&(-1)^{\ell}\hf{k+1}(t,s)\hf{\ell}(s,t)+(-1)^{\ell-1}\int_{s}^{t}\hf{k}\bigl(t,\sigma(\eta)\bigr)\int_{s}^{\eta}\hf{\ell-1}(\zeta,t)\Delta\zeta\Delta\eta\notag
\end{align}
for all $t\in[s,\infty)_{\T}$.
Considering Property~\ref{tspprp1}, we learn that the last term above is nonnegative.
Thus, we have
\begin{align}
(-1)^{\ell}\int_{s}^{t}\hf{k}\bigl(t,\sigma(\eta)\bigr)\hf{\ell}(\eta,t)\Delta\eta
\geq{}&(-1)^{\ell}\hf{k+1}(t,s)\hf{\ell}(s,t)
\geq\hf{k+1}(t,s)\hf{\ell}(t,s)\notag\\
\geq{}&\hf{k+\ell+1}(t,s)\notag
\end{align}
for all $t\in[s,\infty)_{\T}$.
Note that we have applied Lemma~\ref{pilm1} and Lemma~\ref{pilm2} in the first and the second steps above, respectively.
Thus, this completes the proof.
\end{proof}


Now, we have prepared all tools required for the proof of Theorem~\ref{pithm1}.


\begin{proof}[{\proofname} of {\protect\hyperref[pithm1]{dynamic Philos' inequality}}]
Using \hyperref[tsplm1]{Taylor's formula}, Lemma~\ref{tllm1}\,\ref{tllm1it1} and Property~\ref{tspprp1}, we have
\begin{align}
f(t)={}&\sum_{k=0}^{m-1}\hf{k}(t,s)f^{\Delta^{k}}(s)
+\int_{s}^{t}\hf{m-1}\bigl(t,\sigma(\eta)\bigr)f^{\Delta^{m}}(\eta)\Delta\eta\notag\\
\geq{}&\int_{s}^{t}\hf{m-1}\bigl(t,\sigma(\eta)\bigr)f^{\Delta^{m}}(\eta)\Delta\eta\label{pithm1prfeq1}
\end{align}
for all $t\in[s,\infty)_{\T}$.
Noting that $(n-m-1)$ is even, we obtain by Lemma~\ref{tllm1}\,\ref{tllm1it2} that
\begin{align}
f^{\Delta^{m}}(s)={}&\sum_{k=0}^{n-m-1}\hf{k}(s,t)f^{\Delta^{m+k}}(t)
+\int_{t}^{s}\hf{n-m-1}\bigl(s,\sigma(\eta)\bigr)f^{\Delta^{n}}(\eta)\Delta\eta\notag\\
={}&\sum_{k=0}^{n-m-1}(-1)^{k}\hf{k}(s,t)(-1)^{k}f^{\Delta^{m+k}}(t)\notag\\
&+\int_{s}^{t}\hf{n-m-1}\bigl(s,\sigma(\eta)\bigr)\bigl(-f^{\Delta^{n}}(\eta)\bigr)\Delta\eta\notag\\
\geq{}&\hf{n-m-1}(s,t)f^{\Delta^{n-1}}(t)\label{pithm1prfeq2}
\end{align}
for all $t\in[s,\infty)_{\T}$.
Substituting \eqref{pithm1prfeq2} into \eqref{pithm1prfeq1} gives us
\begin{equation}
f(t)\geq\biggl(\int_{s}^{t}\hf{m-1}\bigl(t,\sigma(\eta))\hf{n-m-1}(\eta,t)\Delta\eta\biggr)f^{\Delta^{n-1}}(t)
\quad\text{for}\ t\in[s,\infty)_{\T},\notag
\end{equation}
which completes the proof by an application of Lemma~\ref{pilm3}.
\end{proof}


Now, we have the following corollary of \hyperref[pithm1]{dynamic Philos' inequality}.


\begin{corollary}\label{picrl1}
Assume that $\sup\T=\infty$, $n\in\N$ and $f\in\crd{n}([t_{0},\infty)_{\T},\R_{0}^{+})$ with $f^{\Delta^{n}}\leq0$ on $[t_{0},\infty)_{\T}$.
If $\lim_{t\to\infty}f(t)\neq0$, then for every $\lambda\in(0,1)_{\R}$ there exists $r\in[s,\infty)_{\T}$ such that
\begin{equation}
f(t)\geq\lambda\hf{n-1}(t,t_{0})f^{\Delta^{n-1}}(t)
\quad\text{for all}\ t\in[r,\infty)_{\T},\notag
\end{equation}
where $s\in[t_{0},\infty)_{\T}$ is defined as in \hyperref[tllm1]{Kiguradze's lemma}.
\end{corollary}

\begin{proof}
If $m\in[1,n)_{\Z}$, then the proof follows from \hyperref[pithm1]{dynamic Philos' inequality} since $\hf{n-1}(\cdot,t_{0})\sim\hf{n-1}(\cdot,s)$, i.e., $\lim_{t\to\infty}\frac{\hf{n-1}(t,t_{0})}{\hf{n-1}(t,s)}=1$.
To complete the proof, we consider the case where $m=0$.
This case is possible only when $n\in\N$ is odd.
Let $L:=\lim_{t\to\infty}f(t)$.
Since $L>0$ by the assumption, for any $\lambda\in(0,1)_{\R}$ (if and only if $\sqrt{\lambda}\in(0,1)_{\R}$),
we may find $r\in[s,\infty)_{\T}$ such that $f(r)\leq\frac{L}{\sqrt{\lambda}}$ and $\frac{\hf{n-1}(t,t_{0})}{\hf{n-1}(t,r)}\leq\frac{1}{\sqrt{\lambda}}$ for all $t\in[r,\infty)_{\T}$.
Then, we have
\begin{equation}
f(r)\geq{}f(t)\geq{}L\geq\sqrt{\lambda}f(r)
\quad\text{for all}\ t\in[r,\infty)_{\T}\label{picrl1prfeq1}
\end{equation}
and
\begin{equation}
\hf{n-1}(t,r)\geq\sqrt{\lambda}\hf{n-1}(t,t_{0})
\quad\text{for all}\ t\in[r,\infty)_{\T}.\label{picrl1prfeq2}
\end{equation}
Since $(n-1)$ is even, it follows from \eqref{pithm1prfeq2} and Lemma~\ref{pilm1} that
\begin{equation}
f(r)\geq\hf{n-1}(r,t)f^{\Delta^{n-1}}(t)\geq\hf{n-1}(t,r)f^{\Delta^{n-1}}(t)
\quad\text{for all}\ t\in[r,\infty)_{\T},\notag
\end{equation}
which yields by combining with \eqref{picrl1prfeq1} and \eqref{picrl1prfeq2} that
\begin{align}
f(t)\geq{}&\sqrt{\lambda}f(r)
\geq\sqrt{\lambda}\hf{n-1}(t,r)f^{\Delta^{n-1}}(t)\notag\\
\geq{}&\lambda\hf{n-1}(t,t_{0})f^{\Delta^{n-1}}(t)\notag
\end{align}
for all $t\in[r,\infty)_{\T}$.
This completes the proof.
\end{proof}


\subsection{Recent results}\label{rr}


In this subsection, we give some recent results on delay dynamic equations of higher order.
Consider the delay dynamic inequality
\begin{equation}
x^{\Delta^{n}}(t)+B(t)x\bigl(\beta(t)\bigr)\leq0
\quad\text{for}\ t\in[t_{0},\infty)_{\T}\label{deeq1}
\end{equation}
and the corresponding equation
\begin{equation}
x^{\Delta^{n}}(t)+B(t)x\bigl(\beta(t)\bigr)=0
\quad\text{for}\ t\in[t_{0},\infty)_{\T}.\label{deeq2}
\end{equation}


To be able to extract the next corollary from the following theorem quoted from \cite{MR3091942},
we will give it below with a corrected proof.


\begin{theorem}[\protect{\cite[Theorem~1]{MR3091942}}]\label{drthm1}
The following statements are equivalent.
\begin{enumerate}[label={(\roman*)},leftmargin={*},align={left},ref={(\roman*)}]
\setcounter{enumi}{0}
\item\label{drthm1it1} The inequality \eqref{deeq1} has an eventually positive solution.
\item\label{drthm1it2} The equation \eqref{deeq2} is nonoscillatory.
\end{enumerate}
\end{theorem}

\begin{proof}
The proof will be completed if we can show that \ref{drthm1it1}$\Rightarrow$\ref{drthm1it2} since the implication \ref{drthm1it2}$\Rightarrow$\ref{drthm1it1} is obvious.
Let $x$ be an eventually positive solution of \eqref{deeq1},
then there exists $t_{1}\in[t_{0},\infty)_{\T}$ such that $x(t),x\bigl(\beta(t)\bigr)>0$ for all $t\in[t_{1},\infty)_{\T}$.
An application of \hyperref[tllm1]{Kiguradze's lemma} ensures existence of $m\in[0,n)_{\Z}$ with $(n+m)$ odd and $t_{2}\in[t_{1},\infty)_{\T}$
such that $t\in[t_{2},\infty)_{\T}$ implies $x^{\Delta^{k}}(t)>0$ for all $k\in[0,m)_{\Z}$ and $(-1)^{m+k}x^{\Delta^{k}}(t)>0$ for all $k\in[m,n)_{\Z}$.
Integrating \eqref{deeq1} over $[t,\infty)_{\T}\subset[t_{2},\infty)_{\T}$ for a total of $(n-m-1)$ times, we get
\begin{equation}
x^{\Delta^{m+1}}(t)\geq\int_{t}^{\infty}\hf{n-m-2}\bigl(t,\sigma(\eta)\bigr)B(\eta)x\bigl(\beta(\eta)\bigr)\Delta\eta
\quad\text{for all}\ t\in[t_{2},\infty)_{\T}\notag
\end{equation}
by using Corollary~\ref{tlcrl1} (see \cite[Theorem~3.1]{MR2557101}).
Integrating this over $[t,\infty)_{\T}\subset[t_{2},\infty)_{\T}$, we get
\begin{equation}
x^{\Delta^{m}}(t)\geq{}L+\int_{t}^{\infty}\hf{n-m-1}\bigl(t,\sigma(\eta)\bigr)B(\eta)x\bigl(\beta(\eta)\bigr)\Delta\eta
\quad\text{for all}\ t\in[t_{2},\infty)_{\T},\notag
\end{equation}
where
\begin{equation}
L:=\lim_{t\to\infty}x^{\Delta^{m}}(t).\notag
\end{equation}
By \hyperref[tsplm1]{Taylor's formula}, for all $t\in[t_{2},\infty)_{\T}$, we have
\begin{align}
x(t)={}&\sum_{k=0}^{m-1}\hf{k}(t,t_{2})x^{\Delta^{k}}(t_{2})
+\int_{t_{2}}^{t}\hf{m-1}\bigl(t,\sigma(\eta)\bigr)x^{\Delta^{m}}(\eta)\Delta\eta\notag\\
\geq{}&\sum_{k=0}^{m-1}\hf{k}(t,t_{2})x^{\Delta^{k}}(t_{2})
+\int_{t_{2}}^{t}\hf{m-1}\bigl(t,\sigma(\eta)\bigr)x^{\Delta^{m}}(\eta)\Delta\eta\notag\\
\geq{}&\sum_{k=0}^{m-1}\hf{k}(t,t_{2})x^{\Delta^{k}}(t_{2})\notag\\
&+\int_{t_{2}}^{t}\hf{m-1}\bigl(t,\sigma(\eta)\bigr)
\biggl[L+\int_{\eta}^{\infty}\hf{n-m-1}\bigl(\eta,\sigma(\zeta)\bigr)B(\zeta)x\bigl(\beta(\zeta)\bigr)\Delta\zeta\biggr]\Delta\eta\notag\\
={}&z(t)+\int_{t_{2}}^{t}\hf{m-1}\bigl(t,\sigma(\eta)\bigr)\int_{\eta}^{\infty}\hf{n-m-1}\bigl(\eta,\sigma(\zeta)\bigr)B(\zeta)x\bigl(\beta(\zeta)\bigr)\Delta\zeta\Delta\eta,\notag
\end{align}
where
\begin{equation}
z(t):=\sum_{k=0}^{m-1}\hf{k}(t,t_{2})x^{\Delta^{k}}(t_{2})+L\hf{m}(t,t_{2})\quad\text{for}\ t\in[t_{2},\infty)_{\T}.\notag
\end{equation}
Define
\begin{equation}
\Omega:=\{y\in\cnt{}([t_{2},\infty)_{\T},\R_{0}^{+}):\ x\geq{}y\geq{}z
\quad\text{on}\ [t_{2},\infty)_{\T}\}\notag
\end{equation}
and
\begin{equation}
(\Gamma{}y)(t):=
\begin{cases}
(\Gamma{}y)(t_{3}),&t\in[t_{2},t_{3})_{\T}\\
\begin{aligned}[]
z(t)+&\int_{t_{2}}^{t}\hf{m-1}\bigl(t,\sigma(\eta)\bigr)\\
&\ \times\int_{\eta}^{\infty}\hf{n-m-1}\bigl(\eta,\sigma(\zeta)\bigr)B(\zeta)y\bigl(\beta(\zeta)\bigr)\Delta\zeta\Delta\eta,
\end{aligned}
&\begin{aligned}[]
\phantom{\int_{t_{2}}^{t}}\\
t\in[t_{3},\infty)_{\T},
\end{aligned}
\end{cases}\notag
\end{equation}
where $t_{3}\in[t_{2},\infty)_{\T}$ satisfies $\beta(t_{3})\geq{}t_{2}$.
Define a sequence of functions $\{y_{k}\}_{k\in\N_{0}}\subset\Omega$ by $y_{k}:=\Gamma{}y_{k-1}$ for $k\in\N$ and $y_{0}:=z$.
It is clear that $\{y_{k}\}_{k\in\N_{0}}$ is a nondecreasing sequence of functions bounded above by $x$.
Define $y:=\lim_{k\to\infty}y_{k}$,
then we see that $y=\Gamma{}y$ on $[t_{2},\infty)_{\T}$,
which is a nonoscillatory solution of \eqref{deeq2}.
Note that $y$ satisfies $y^{\Delta^{k}}(t)>0$ for all $k\in[0,m)_{\Z}$ and $(-1)^{m+k}y^{\Delta^{k}}(t)>0$ for all $k\in[m,n)_{\Z}$.
This completes the proof.
\end{proof}


\begin{corollary}[\protect{\cite[Corollary~1]{MR3091942}}]\label{drcrl1}
The following statements are equivalent.
\begin{enumerate}[label={(\roman*)},ref={(\roman*)},leftmargin={0pt},itemindent={*},noitemsep,topsep=0pt,parsep=0pt,partopsep=0pt]
\setcounter{enumi}{0}
\item\label{drcrl1it1} The inequality \eqref{deeq1} has an eventually positive solution, which does not tend to zero asymptotically.
\item\label{drcrl1it2} The equation \eqref{deeq2} has a nonoscillatory solution, which does not tend to zero asymptotically.
\end{enumerate}
\end{corollary}


\begin{theorem}[\protect{\cite[Theorem~2]{MR3091942}}]\label{drthm2}
Assume that \ref{r1} holds.
\begin{enumerate}[label={(\roman*)},ref={(\roman*)},leftmargin={0pt},itemindent={*},noitemsep,topsep=0pt,parsep=0pt,partopsep=0pt]
\setcounter{enumi}{0}
\item\label{drthm2it1} If $n\in\N$ is even and \eqref{introeq1} has a nonoscillatory solution, then so does
    \begin{equation}
    x^{\Delta^{n}}(t)+\bigl[1-A\bigl(\beta(t)\bigr)\bigr]B(t)x\bigl(\beta(t)\bigr)=0
    \quad\text{for}\ t\in[t_{0},\infty)_{\T}.\label{drthm2eq1}
    \end{equation}
\item\label{drthm2it2} If $n\in\N$ is odd and \eqref{introeq1} has a nonoscillatory solution,
    which does not tend to zero at infinity, then so does \eqref{drthm2eq1}.
\end{enumerate}
\end{theorem}


\begin{theorem}[\protect{\cite[Theorem~3]{MR3091942}}]\label{drthm3}
Assume that $n\in\N$ and \ref{r2} holds.
If \eqref{introeq1} has a nonoscillatory solution,
which does not tend to zero at infinity,
then so does \eqref{deeq2}.
\end{theorem}


\section{Main results}\label{mr}


\subsection{Nonneutral equations}\label{nne}


We continue our discussion with nonneutral differential equations.
We first consider even-order dynamic equations.


\begin{theorem}\label{nnethm1}
Assume that $n\in\N$ is even.
If there exists $\lambda\in(0,1)_{\R}$ such that the first-order delay dynamic equation
\begin{equation}
x^{\Delta}(t)+\lambda{}B(t)\hf{n-1}\bigl(\beta(t),t_{0}\bigr)x\bigl(\beta(t)\bigr)=0\quad\text{for}\ t\in[t_{0},\infty)_{\T}\label{nnethm1eq1}
\end{equation}
is oscillatory, then \eqref{deeq2} is also oscillatory.
\end{theorem}

\begin{proof}
Assume, on the contrary, that $x$ is an eventually positive solution of \eqref{deeq2}.
Then, there exists $t_{1}\in[t_{0},\infty)_{\T}$ such that $x(t),x\bigl(\beta(t)\bigr)>0$ for all $t\in[t_{1},\infty)_{\T}$.
By \hyperref[tllm1]{Kiguradze's lemma}, we learn that there exist $t_{2}\in[t_{1},\infty)_{\T}$ and
$m\in[0,n)_{2\Z-1}$ such that for all $t\in[t_{2},\infty)_{\T}$,
we have $x^{\Delta^{k}}(t)>0$ for all $k\in[0,m)_{\Z}$ and $(-1)^{m+k}x^{\Delta^{k}}(t)>0$ for all $k\in[m,n)_{\Z}$.
In particular, $x$ is positive and increasing on $[t_{2},\infty)_{\T}$.
Using Corollary~\ref{picrl1}, we get for $\lambda\in(0,1)_{\R}$ that
\begin{align}
x(t)\geq\lambda{}\hf{n-1}(t,t_{0})x^{\Delta^{n-1}}(t)\quad\text{for all}\ t\in[t_{3},\infty)_{\T}\label{nnethm1prfeq1}
\end{align}
for some $t_{3}\in[t_{2},\infty)_{\T}$.
Substituting \eqref{nnethm1prfeq1} into \eqref{deeq2},
and using the nondecreasing nature of $x(\beta(\cdot))$ ($x$ is increasing and $\beta$ is nondecreasing),
we obtain
\begin{equation}
x^{\Delta^{n}}(t)+\lambda{}B(t)\hf{n-1}\bigl(\beta(t),t_{0}\bigr)x^{\Delta^{n-1}}\bigl(\beta(t)\bigr)\leq0\quad\text{for all}\ t\in[t_{4},\infty)_{\T},\label{nnethm1prfeq2}
\end{equation}
where $t_{4}\in[t_{3},\infty)_{\T}$ satisfies $\beta(t_{4})\geq{}t_{3}$.
Note that $x^{\Delta^{n-1}}$ is positive on $[t_{4},\infty)_{\T}$ and satisfies
\begin{equation}
y^{\Delta}(t)+\lambda{}B(t)\hf{n-1}\bigl(\beta(t),t_{0}\bigr)y\bigl(\beta(t)\bigr)\leq0\quad\text{for all}\ t\in[t_{4},\infty)_{\T},\notag
\end{equation}
which is a contradiction since \eqref{nnethm1eq1} also has an eventually positive solution by Theorem~\ref{drthm1} (see also \cite[Theorem~3.1 and Corollary~4.2]{MR2683912}).
This completes the proof.
\end{proof}


Combining Theorem~\ref{nnethm1} with \cite{MR2475963} and \cite{MR3518254} yields the following corollary.


\begin{corollary}\label{nnecrl1}
Assume that $n\in\N$ is even.
If
\begin{equation}
\liminf_{t\to\infty}\inf_{\substack{-\lambda{}B\hf{n-1}(\beta(\cdot),t_{0})\in\reg{+}([\beta(t),t)_{\T})\\ \lambda>0}}\biggl\{\frac{1}{\lambda\ef{-\lambda{}B\hf{n-1}(\beta(\cdot),t_{0})}\bigl(t,\beta(t)\bigr)}\biggr\}>1,\label{nnecrl1eq1}
\end{equation}
or
\begin{equation}
\liminf_{t\to\infty}\int_{\beta(t)}^{t}B(\eta)\hf{n-1}\bigl(\beta(\eta),t_{0}\bigr)\Delta\eta>\gamma\label{nnecrl1eq2}
\end{equation}
and
\begin{equation}
\limsup_{t\to\infty}\int_{\beta(t)}^{\sigma(t)}B(\eta)\hf{n-1}\bigl(\beta(\eta),t_{0}\bigr)\Delta\eta
>1-\Bigl(1-\sqrt{1-\gamma}\Bigr)^{2},\label{nnecrl1eq3}
\end{equation}
then every solution of \eqref{introeq1} oscillates.
\end{corollary}


We next consider odd-order dynamic equations.


\begin{theorem}\label{nnethm2}
Assume that $n\in\N$ is odd and
\begin{equation}
\int_{t_{0}}^{\infty}B(\eta)\hf{n-1}\bigl(t_{0},\sigma(\eta)\bigr)\Delta\eta=\infty.\label{nnethm2eq1}
\end{equation}
If there exists $\lambda\in(0,1)_{\R}$ such that the first-order delay dynamic equation \eqref{nnethm1eq1}
is oscillatory,
then every solution of \eqref{deeq2} is oscillatory or tends to zero asymptotically.
\end{theorem}

\begin{proof}
Assume, on the contrary, that $x$ is an eventually positive solution of \eqref{deeq2},
which asymptotically does not tend to zero.
Then, there exists $t_{1}\in[t_{0},\infty)_{\T}$ such that $x(t),x\bigl(\beta(t)\bigr)>0$ for all $t\in[t_{1},\infty)_{\T}$.
By \hyperref[tllm1]{Kiguradze's lemma}, we learn that there exist $t_{2}\in[t_{1},\infty)_{\T}$ and $m\in[0,n)_{2\Z}$ such that for all $t\in[t_{2},\infty)_{\T}$,
we have $x^{\Delta^{k}}(t)>0$ for all $k\in[0,m)_{\Z}$ and $(-1)^{m+k}x^{\Delta^{k}}(t)>0$ for all $k\in[m,n)_{\Z}$.
We have the following two possible cases.
\begin{enumerate}[label={(C\arabic*)},ref={(C\arabic*)},leftmargin={0pt},itemindent={*},noitemsep,topsep=0pt,parsep=0pt,partopsep=0pt]
\setcounter{enumi}{0}
\item If $m\in[2,n)_{2\Z}$, then we proceed as in the proof of Theorem~\ref{nnethm1} and arrive at a contradiction.
\item If $m=0$, then we learn that $x$ is bounded, thus it follows from \cite[Theorem~3.1]{MR2557101} that \eqref{nnethm2eq1} implies $\lim_{t\to\infty}x(t)=0$,
    which is also a contradiction.
\end{enumerate}
The proof is therefore complete.
\end{proof}


Combining Theorem~\ref{nnethm2} with \cite{MR2475963} and \cite{MR3518254} yields the following corollary.


\begin{corollary}
Assume that $n\in\N$ is odd and \eqref{nnethm2eq1} holds.
If \eqref{nnecrl1eq1}, or \eqref{nnecrl1eq2} and \eqref{nnecrl1eq3},
then every solution of \eqref{introeq1} oscillates or tends to zero asymptotically.
\end{corollary}


\begin{example}\label{nneexm1}
Let $\T=q^{\Z}\cup\{0\}$, where $q\in(1,\infty)_{\R}$, and consider the $q$-difference equation
\begin{equation}
\mathrm{D}_{q}^{n}x(t)+\frac{b_{0}}{t^{n}}x(t/q^{\beta_{0}})=0
\quad\text{for}\ t\in{}q^{\N},\label{nneexm1eq1}
\end{equation}
where $n\in\N$, $b_{0}\in\R^{+}$ and $\beta_{0}\in\N$.
Remark~\ref{pirmk1} and
\begin{equation}
\int_{1}^{\infty}\frac{b_{0}}{\eta}\Delta\eta=\infty,\notag
\end{equation}
readily imply \eqref{nnethm2eq1}.
We compute
\begin{equation}
\begin{aligned}[]
\int_{t/q^{\beta_{0}}}^{t}\frac{b_{0}}{\eta^{n}}\hf{n-1}\bigl(\eta/q^{\beta_{0}},1\bigr)\Delta\eta
={}&\int_{t/q^{\beta_{0}}}^{t}\frac{b_{0}}{\eta^{n}}\frac{\hf{n-1}\bigl(\eta/q^{\beta_{0}},1\bigr)}{(\eta/q^{\beta_{0}})^{n-1}}\biggl(\frac{\eta}{q^{\beta_{0}}}\bigg)^{n-1}\Delta\eta\\
\sim{}&\frac{1}{q^{\beta_{0}(n-1)}\Gamma_{q}(n-1)}\int_{t/q^{\beta_{0}}}^{t}\frac{b_{0}}{\eta}\Delta\eta\\
={}&\frac{(q-1)b_{0}\beta_{0}}{q^{\beta_{0}(n-1)}\Gamma_{q}(n-1)}
\end{aligned}\notag
\end{equation}
for $t\in{}q^{\N}$.
In view of \cite[Example~3.3]{MR2475963}, \eqref{nnecrl1eq1} reduces to
\begin{equation}
\frac{(q-1)b_{0}\beta_{0}}{q^{\beta_{0}(n-1)}\Gamma_{q}(n-1)}
>\biggl(\frac{\beta_{0}}{\beta_{0}+1}\biggr)^{\beta_{0}+1}.\label{nneexm1eq2}
\end{equation}
Hence, if \eqref{nneexm1eq2} holds, then every solution of \eqref{nneexm1eq1} oscillates when $n$ is even while
oscillates or tends to zero asymptotically when $n$ is odd.
\end{example}


\subsection{Neutral equations}\label{ne}


In this subsection, we extend our results to higher-order neutral dynamic equations.
First two theorems here consider the first range \ref{r1}.


\begin{theorem}\label{nethm1}
Assume that $n\in\N$ is even and \ref{r1} hold.
Moreover, assume that there exists $\lambda\in(0,1)_{\R}$ such that the first-order delay dynamic equation
\begin{equation}
x^{\Delta}(t)+\lambda\bigl[1-A\bigl(\beta(t)\bigr)\bigr]B(t)\hf{n-1}\bigl(\beta(t),t_{0}\bigr)x\bigl(\beta(t)\bigr)=0\quad\text{for}\ t\in[t_{0},\infty)_{\T}\label{nethm1eq1}
\end{equation}
is oscillatory.
Then, \eqref{introeq1} is also oscillatory.
\end{theorem}

\begin{proof}
Assume, on the contrary, that \eqref{introeq1} has a nonoscillatory solution.
Then, by Theorem~\ref{drthm2}\,\ref{drthm2it1}, \eqref{drthm2eq1} also has a nonoscillatory solution.
Without loss of generality, assume that $x$ is an eventually positive solution of \eqref{drthm2eq1}.
There exists $t_{1}\in[t_{0},\infty)_{\T}$ such that $x(t),x(\alpha(t)),x(\beta(t))>0$ for all $t\in[t_{1},\infty)_{\T}$.
It follows from \hyperref[tllm1]{Kiguradze's lemma} that there exist $t_{2}\in[t_{1},\infty)_{\T}$ and
$m\in[0,n)_{2\Z-1}$ such that for all $t\in[t_{2},\infty)_{\T}$,
we have $x^{\Delta^{k}}(t)>0$ for all $k\in[0,m)_{\Z}$ and $(-1)^{m+k}x^{\Delta^{k}}(t)>0$ for all $k\in[m,n)_{\Z}$.
In particular, $x$ is positive and increasing on $[t_{2},\infty)_{\T}$.
By Corollary~\ref{picrl1}, \eqref{nnethm1prfeq1} holds for all $t\in[t_{3},\infty)_{\T}$, where $t_{3}\in[t_{2},\infty)_{\T}$.
Substituting \eqref{nnethm1prfeq1} into \eqref{introeq1}, and using the nondecreasing nature of $x(\beta(\cdot))$,
we obtain
\begin{equation}
x^{\Delta^{n}}(t)+\lambda{}\bigl[1-A\bigl(\beta(t)\bigr)\bigr]B(t)\hf{n-1}\bigl(\beta(t),t_{0}\bigr)x^{\Delta^{n-1}}\bigl(\beta(t)\bigr)\leq0
\quad\text{for all}\ t\in[t_{4},\infty)_{\T},\notag
\end{equation}
where $t_{4}\in[t_{3},\infty)_{\T}$ satisfies $\beta(t_{4})\geq{}t_{3}$.
Note that $x^{\Delta^{n-1}}$ is positive on $[t_{4},\infty)_{\T}$ and satisfies
\begin{equation}
y^{\Delta}(t)+\lambda{}\bigl[1-A\bigl(\beta(t)\bigr)\bigr]B(t)\hf{n-1}\bigl(\beta(t),t_{0}\bigr)y\bigl(\beta(t)\bigr)\leq0
\quad\text{for all}\ t\in[t_{4},\infty)_{\T}.\notag
\end{equation}
By Theorem~\ref{drthm1} (see also \cite[Theorem~3.1 and Corollary~4.2]{MR2683912}),
this implies that \eqref{nnethm1eq1} also has an eventually positive solution.
This is a contradiction and the proof is complete.
\end{proof}


\begin{corollary}\label{necrl1}
Assume that $n\in\N$ is even and \ref{r1} hold.
If
\begin{equation}
\liminf_{t\to\infty}\inf_{\substack{-\lambda[1-A(\beta(\cdot))]B\hf{n-1}(\beta(\cdot),t_{0})\in\reg{+}\\ \lambda>0}}\biggl\{\frac{1}{\lambda\ef{-\lambda[1-A(\beta(\cdot))]B\hf{n-1}(\beta(\cdot),t_{0})}\bigl(t,\beta(t)\bigr)}\biggr\}>1,\label{necrl1eq1}
\end{equation}
or
\begin{equation}
\liminf_{t\to\infty}\int_{\beta(t)}^{t}\bigl[1-A\bigl(\beta(\eta)\bigr)\bigr]B(\eta)\hf{n-1}\bigl(\beta(\eta),t_{0}\bigr)\Delta\eta>\gamma\label{necrl1eq2}
\end{equation}
and
\begin{equation}
\limsup_{t\to\infty}\int_{\beta(t)}^{\sigma(t)}\bigl[1-A\bigl(\beta(\eta)\bigr)\bigr]B(\eta)\hf{n-1}\bigl(\beta(\eta),t_{0}\bigr)\Delta\eta
>1-\Bigl(1-\sqrt{1-\gamma}\Bigr)^{2},\label{necrl1eq3}
\end{equation}
then every solution of \eqref{introeq1} oscillates.
\end{corollary}


We would like to mention that Theorem~\ref{nethm1} includes \cite[Theorem~1]{MR2250359}.


\begin{example}\label{neexm1}
Let $\T=\Z$ and consider the difference equation
\begin{equation}
\Delta^{n}[x(t)+a_{0}x(t-\alpha_{0})]+\frac{b_{0}}{t^{p}}x(t-\beta_{0})=0
\quad\text{for}\ t\in\N_{0},\label{neexm1eq1}
\end{equation}
where $n\in\N$ is even, $a_{0}\in(0,1)_{\R}$, $b_{0}\in\R^{+}$, $p\in\R_{0}^{+}$, $\alpha_{0},\beta_{0}\in\N$.
By \cite[Theorem~3\,(i)]{MR1635228}, \eqref{neexm1eq1} is oscillatory if $p\leq1$.
By \cite[Theorem~1\,(a)]{MR1429478}, \eqref{neexm1eq1} is oscillatory if $p<n-1$, or
\begin{equation}
p=n-1
\quad\text{and}\quad
b_{0}(1-a_{0})>\frac{(2^{n-1})^{(n-1)}}{(n-1)!}\frac{\beta_{0}^{\beta_{0}}}{(\beta_{0}+1)^{\beta_{0}+1}},\notag
\end{equation}
where $^{(\cdot)}$ denotes the falling factorial function.
Applying Corollary~\ref{necrl1} to \eqref{neexm1eq1} drops the factor $(2^{n-1})^{(n-1)}$ above (see Remark~\ref{pirmk2}), i.e.,
$p<n-1$, or
\begin{equation}
p=n-1
\quad\text{and}\quad
b_{0}(1-a_{0})>\frac{1}{(n-1)!}\frac{\beta_{0}^{\beta_{0}}}{(\beta_{0}+1)^{\beta_{0}+1}}\notag
\end{equation}
implies oscillation of all solutions of \eqref{neexm1eq1}.
\end{example}


Before we proceed to the next theorem, we would like to remark that \ref{r1} establishes equivalence between divergence of the integrals
\begin{equation}
\int_{t_{0}}^{\infty}B(\eta)\hf{n-1}\bigl(t_{0},\sigma(\eta)\bigr)\Delta\eta
\quad\text{and}\quad
\int_{t_{0}}^{\infty}\bigl[1-A\bigl(\beta(t)\bigr)\bigr]B(\eta)\hf{n-1}\bigl(t_{0},\sigma(\eta)\bigr)\Delta\eta.\notag
\end{equation}


\begin{theorem}\label{nethm2}
Assume that $n\in\N$ is odd, \ref{r1} and \eqref{nnethm2eq1} hold.
Moreover, assume that there exists $\lambda\in(0,1)_{\R}$ such that the first-order delay dynamic equation \eqref{nethm1eq1} is oscillatory.
Then, every solution of \eqref{introeq1} oscillates or tends to zero asymptotically.
\end{theorem}

\begin{proof}
Assume the contrary that \eqref{introeq1} admits a nonoscillatory solution,
which asymptotically does not tend to zero.
By Theorem~\ref{drthm2}\,\ref{drthm2it2}, \eqref{drthm2eq1} also has a solution of the same kind.
Without loss of generality, assume that $x$ is an eventually positive solution of \eqref{drthm2eq1},
which does not tend to zero at infinity.
Then, $x(t),x(\alpha(t)),x(\beta(t))>0$ for all $t\in[t_{1},\infty)_{\T}$, where $t_{1}\in[t_{0},\infty)_{\T}$.
It follows from \hyperref[tllm1]{Kiguradze's lemma} that there exist $t_{2}\in[t_{1},\infty)_{\T}$ and
$m\in[0,n)_{2\Z}$ such that for all $t\in[t_{2},\infty)_{\T}$,
we have $x^{\Delta^{k}}(t)>0$ for all $k\in[0,m)_{\Z}$ and $(-1)^{m+k}x^{\Delta^{k}}(t)>0$ for all $k\in[m,n)_{\Z}$.
We have the following two possible cases.
\begin{enumerate}[label={(C\arabic*)},ref={(C\arabic*)},leftmargin={0pt},itemindent={*},noitemsep,topsep=0pt,parsep=0pt,partopsep=0pt]
\setcounter{enumi}{0}
\item If $m\in[2,n)_{2\Z}$, then we proceed as in the proof of Theorem~\ref{nethm1} and arrive at a contradiction.
\item If $m=0$, then $x$ is positive and decreasing, i.e., $x$ is bounded.
    By virtue of \cite[Theorem~3.1]{MR2557101}, $\lim_{t\to\infty}x(t)=0$.
    This is a contradiction.
\end{enumerate}
The proof is therefore complete.
\end{proof}

\begin{corollary}
Assume that $n\in\N$ is odd, \ref{r1} and \eqref{nnethm2eq1} hold.
If \eqref{necrl1eq1}, or \eqref{necrl1eq2} and \eqref{necrl1eq3},
then every solution of \eqref{introeq1} oscillates or tends to zero asymptotically.
\end{corollary}


The following remark can be extracted from the first part of the proof of the above theorem.


\begin{remark}
Under the conditions of Theorem~\ref{nethm2} except \eqref{nnethm2eq1},
we can prove that every unbounded solution of \eqref{introeq1} oscillates.
\end{remark}


The final result of this section focuses on the latter range \ref{r2}.


\begin{theorem}\label{nethm3}
Assume that $n\in\N$, \ref{r2} and \eqref{nnethm2eq1} hold.
Moreover, assume that there exists $\lambda\in(0,1)_{\R}$ such that the first-order delay dynamic equation \eqref{nnethm1eq1} is oscillatory.
Then, every solution of \eqref{introeq1} oscillates or tends to zero asymptotically.
\end{theorem}

\begin{proof}
The proof follows by using similar arguments to that in the proofs of Theorem~\ref{nethm1} and Theorem~\ref{nethm2}
but in that case Theorem~\ref{drthm3} should be applied instead of Theorem~\ref{drthm2}.
Thus, the details of the proof are omitted.
\end{proof}


\begin{corollary}
Assume that $n\in\N$, \ref{r2} and \eqref{nnethm2eq1} hold.
If \eqref{nnecrl1eq1}, or \eqref{nnecrl1eq2} and \eqref{nnecrl1eq3},
then every solution of \eqref{introeq1} oscillates or tends to zero asymptotically.
\end{corollary}


\begin{example}[See \protect{\cite[Example~3]{MR2040224}}]\label{neexm2}
Let $\T=\R$, and $n\in\N$ be even.
Consider
\begin{equation}
\bigg[x(t)-\frac{1-\sin(t)}{3}x(t/\alpha_{0})\bigg]^{(n)}+\frac{b_{0}}{t^{n}}x(t/\beta_{0})=0\quad\text{for}\ t\in[1,\infty)_{\R},\label{neexm2eq1}
\end{equation}
where $\alpha_{0}\in(1,\infty)_{\R}$, $\beta_{0}\in[1,\infty)_{\R}$ and $b_{0}\in\R^{+}$.
If we apply Theorem~\ref{nethm3}, the corresponding first-order differential equation is
\begin{equation}
x^{\prime}(t)+\lambda\frac{b_{0}}{\beta_{0}^{n-1}(n-1)!t}x(t/\beta_{0})=0\quad\text{for}\ t\in[1,\infty)_{\R},\label{neexm2eq2}
\end{equation}
where $\lambda\in(0,1)_{\R}$,
which is oscillatory if
\begin{equation}
\frac{b_{0}\ln(\beta_{0})}{\beta_{0}^{n-1}(n-1)!}>\frac{1}{\mathrm{e}}.\notag
\end{equation}
By \cite[Theorem~2, Corollary~5]{MR2040224}, all solutions to \eqref{neexm2eq1} oscillate if
\begin{equation}
\frac{b_{0}}{\beta_{0}^{n-1}(n-1)!(n-1)}>\frac{1}{4}.\notag
\end{equation}
Thus, Theorem~\ref{nethm3} gives a better result when $\beta_{0}^{n-1}>\exp\bigl\{\frac{4}{\mathrm{e}}\bigr\}$
or equivalently $\beta_{0}>\exp\bigl\{\frac{4}{\mathrm{e}(n-1)}\bigr\}$.
For instance, when $n=4$, we have $\beta_{0}>1.63314$.
\end{example}


\section{Appendix}\label{app}

\subsection{Appendix A: Time scales essentials}\label{tse}


A \emph{time scale}, which inherits the standard topology on $\R$, is a nonempty closed subset of reals.
Here, and later throughout this paper, a time scale will be denoted by the symbol $\T$, and the intervals with a subscript $\T$ are used to denote the intersection of the usual interval with $\T$.
For $t\in\T$, we define the \emph{forward jump operator} $\sigma:\T\to\T$ by $\sigma(t):=\inf(t,\infty)_{\T}$ while the \emph{backward jump operator} $\rho:\T\to\T$ is defined by $\rho(t):=\sup(-\infty,t)_{\T}$, and the \emph{graininess function} $\mu:\T\to\R_{0}^{+}$ is defined to be $\mu(t):=\sigma(t)-t$.

\begin{table}[h!tb]
  \centering
    \begin{tabular}{c|ccc}
    \hline
    $\T$ & $\R$ & $h\Z$, $h\in\R^{+}$ & $q^{\N_{0}}$, $q\in(1,\infty)_{\R}$ \\ \hline
    $\sigma(t)$ & $t$ & $t+h$ & $qt$ \\
    $\rho(t)$ & $t$ & $t-h$ & $t/q$ \\
    $\mu(t)$ & $0$ & $h$ & $(q-1)t$ \\
    \hline
    \end{tabular}
  \caption{The explicit forms of the forward jump, the backward jump and the graininess on some time scales.}\label{tsetbl1}
\end{table}

A point $t\in\T$ is called \emph{right-dense} if $\sigma(t)=t$ and/or equivalently $\mu(t)=0$ holds; otherwise, it is called \emph{right-scattered}, and similarly \emph{left-dense} and \emph{left-scattered} points are defined with respect to the backward jump operator.
For $f:\T\to\R$ and $t\in\T$, the $\Delta$-derivative $f^{\Delta}(t)$ of $f$ at the point $t$ is defined to be the number, provided it
exists, with the property that, for any $\varepsilon>0$, there is a neighborhood $U$ of $t$ such that
\begin{equation}
|[f^{\sigma}(t)-f(s)]-f^{\Delta}(t)[\sigma(t)-s]|\leq\varepsilon|\sigma(t)-s|
\quad\text{for all}\ s\in U,\notag
\end{equation}
where $f^{\sigma}:=f\circ\sigma$ on $\T$.
We mean the $\Delta$-derivative of a function when we only say derivative unless otherwise is specified.

\begin{table}[h!tb]
  \centering
    \begin{tabular}{c|ccc}
    \hline
    $\T$ & $\R$ & $h\Z$, $h\in\R^{+}$ & $q^{\N_{0}}$, $q\in(1,\infty)_{\R}$ \\ \hline
    $f^{\Delta}(t)$ & $f^{\prime}(t)$ & $\dfrac{f(t+h)-f(t)}{h}$ & $\dfrac{f(qt)-f(t)}{(q-1)t}$ \\
    \hline
    \end{tabular}
  \caption{The explicit forms of the delta derivative on some time scales.}\label{tsetbl2}
\end{table}

A function $f$ is called \emph{rd-continuous} provided that it is continuous at right-dense points in
$\T$, and has a finite limit at left-dense points, and the \emph{set of rd-continuous functions}
is denoted by $\crd{}(\T,\R)$.
The set of functions $\crd{1}(\T,\R)$ includes the functions whose derivative is in $\crd{}(\T,\R)$ too.
For a function $f\in\crd{1}(\T,\R)$, the so-called \emph{simple useful formula} holds
\begin{equation}
f^{\sigma}(t)=f(t)+\mu(t)f^{\Delta}(t)
\quad\text{for all}\ t\in\T^{\kappa},\notag
\end{equation}
where $\T^{\kappa}:=\T\backslash\{\sup\T\}$ if $\sup\T<\infty$ and satisfies $\rho(\sup\T)<\sup\T$; otherwise, $\T^{\kappa}:=\T$.
For $s,t\in\T$ and a function $f\in\crd{}(\T,\R)$, the $\Delta$-integral of $f$ is defined by
\begin{equation}
\int_{s}^{t}f(\eta)\Delta\eta=F(t)-F(s)
\quad\text{for}\ s,t\in\T,\notag
\end{equation}
where $F\in\crd{1}(\T,\R)$ is an antiderivative of $f$, i.e., $F^{\Delta}=f$ on $\T^{\kappa}$.

\begin{table}[h!tb]
  \centering
    \begin{tabular}{c|ccc}
    \hline
    $\T$ & $\R$ & $h\Z$, $h\in\R^{+}$ & $q^{\N_{0}}$, $q\in(1,\infty)_{\R}$ \\ \hline
    $\displaystyle\int_{s}^{t}f(\eta)\Delta\eta$ & $\displaystyle\int_{s}^{t}f(\eta)\mathrm{d}\eta$ & $h\displaystyle\sum_{\eta=s/q}^{t/q-1}f(h\eta)$ & $(q-1)\displaystyle\sum_{\eta=\log_{q}(s)}^{\log_{q}(t/q)}f(q^{\eta})q^{\eta}$ \\
    \hline
    \end{tabular}
  \caption{The explicit forms of the delta integral on some time scales.}\label{tsetbl3}
\end{table}

\subsection{Appendix B: Time scales polynomials}\label{tsp}

The generalized polynomials on time scales
(see \cite[Lemma~5]{MR1678096} and/or \cite[\S~1.6]{MR1843232}) $\hf{k}\in\cnt{}(\T\times\T,\R)$ are defined by
\begin{equation}
\hf{k}(t,s):=
\begin{cases}
1,&k=0\\
\displaystyle\int_{s}^{t}\hf{k-1}(\eta,s)\Delta\eta,&k\in\N
\end{cases}\quad\text{for}\ s,t\in\T.\label{tspeq1}
\end{equation}

\begin{table}[h!tb]
  \centering
    \begin{tabular}{c|ccc}
    \hline
    $\T$ & $\R$ & $h\Z$, $h\in\R^{+}$ & $q^{\N_{0}}$, $q\in(1,\infty)_{\R}$ \\ \hline
    $h_{n}(t,s)$ & $\displaystyle\frac{(t-s)^{n}}{n!}$ & $\displaystyle\frac{1}{n!}\prod_{i=0}^{n-1}(t-ih-s)$ & $\displaystyle\prod_{i=0}^{n-1}\frac{t-q^{i}s}{\sum_{j=0}^{i}q^{j}}$ \\
    \hline
    \end{tabular}
  \caption{The explicit forms of the monomials on some time scales.}\label{tsptbl1}
\end{table}

Note that, for all $s,t\in\T$ and all $k\in\N_{0}$, the function $\hf{k}$ satisfies
\begin{equation}
\hf{k}^{\Delta_{1}}(t,s)=
\begin{cases}
0,&k=0\\
\hf{k-1}(t,s),&k\in\N.
\end{cases}\label{tspeq2}
\end{equation}

\begin{property}[\protect{\cite[Property~1]{MR2506155}}]\label{tspprp1}
By using induction and \eqref{tspeq1}, it is easy to see for all $k\in\N_{0}$ that $\hf{k}(\cdot,s)\geq0$ on $[s,\infty)_{\T}$ and $(-1)^{k}\hf{k}(\cdot,s)\geq0$ on $(-\infty,s]_{\T}$.
In view of \eqref{tspeq2}, for all $k\in\N$, $\hf{k}(\cdot,s)$ is increasing on $[s,\infty)_{\T}$,
and $(-1)^{k}\hf{k}(\cdot,s)$ is decreasing on $(-\infty,s]_{\T}$.
\end{property}

\begin{lemma}[Taylor's formula \protect{\cite[Theorem~1.113]{MR1843232}}]\label{tsplm1}
If $n\in\N$, $s\in\T$ and $f\in\crd{n}(\T,\R)$, then
\begin{equation}
f(t)=\sum_{k=0}^{n-1}\hf{k}(t,s)f^{\Delta^{k}}(s)+\int_{s}^{t}\hf{n-1}\bigl(t,\sigma(\eta)\bigr)f^{\Delta^{n}}(\eta)\Delta\eta
\quad\text{for}\ t\in\T.\notag
\end{equation}
\end{lemma}

\begin{lemma}[\protect{\cite[Theorem~4.1]{MR2320804}}]\label{tsplm2}
If $k\in\N$, $\ell\in\N_{0}$ and $s\in\T$, then
\begin{equation}
\hf{k+\ell}(t,s)=\int_{s}^{t}\hf{k-1}\bigl(t,\sigma(\eta)\bigr)\hf{\ell}(\eta,s)\Delta\eta
\quad\text{for}\ t\in\T.\notag
\end{equation}
\end{lemma}

As an immediate consequence of Lemma~\ref{tsplm2},
we can give the following alternative definition of the generalized polynomials:
\begin{equation}
\hf{k}(t,s):=
\begin{cases}
1,&k=0\\
\displaystyle\int_{s}^{t}\hf{k-1}\bigl(t,\sigma(\eta)\bigr)\Delta\eta,&k\in\N
\end{cases}\quad\text{for}\ s,t\in\T.\label{tspeq3}
\end{equation}

\begin{remark}
Using \cite[Theorem~1.112]{MR1843232} in Lemma~\ref{pilm1} yields the inequality
\begin{equation}
\gf{k}(t,s)\geq\hf{k}(t,s)
\quad\text{for}\ t\in[s,\infty)_{\T}\ \text{and}\ k\in\N_{0},\nonumber
\end{equation}
where $\gf{k}\in\cnt{}(\T\times\T,\R)$ is defined by
\begin{equation}
\gf{k}(t,s):=
\begin{cases}
1,&k=0\\
\displaystyle\int_{s}^{t}\gf{k-1}\bigl(\sigma(\eta),s\bigr)\Delta\eta,&k\in\N
\end{cases}\quad\text{for}\ s,t\in\T.\nonumber
\end{equation}
\end{remark}


\begin{thebibliography}{99}

\bibitem{MR1155840} {\sc R.\,P.~Agarwal},
\newblock {\it {D}ifference {E}quations and {I}nequalities. {T}heory, {M}ethods, and {A}pplications},
\newblock Marcel Dekker, Inc., New York, 1992.

\bibitem{MR1678096} {\sc R.\,P.~Agarwal and M.~Bohner},
\newblock {\it {B}asic calculus on time scales and some of its applications},
\newblock {Results Math.} {\bf 35}, 1-2 (1999), 3--22.

\bibitem{MR1429478} {\sc R.\,P.~Agarwal, E.~Thandapani and P.\,J.\,Y.~Wong},
\newblock {\it {O}scillations of higher-order neutral difference equations},
\newblock {Appl.\ Math.\ Lett.} {\bf 10}, 1 (1997), 71--78.

\bibitem{MR2536378} {\sc R.\,P.~Agarwal and M.~Bohner},
\newblock {\it {A}n oscillation criterion for first order delay dynamic equations},
\newblock {Funct.\ Differ.\ Equ.} {\bf 16}, 1 (2009), 11--17.

\bibitem{MR2397023} {\sc H.\,A.~Agwo},
\newblock {\it {O}n the oscillation of first order delay dynamic equations with variable coefficients},
\newblock {Rocky Mountain J.\ Math.} {\bf 38}, 1 (2008), 1--18.

\bibitem{MR1843232} {\sc M.~Bohner and A.~Peterson},
\newblock {\it {D}ynamic {E}quations on {T}ime {S}cales. {A}n {I}ntroduction with {A}pplications},
\newblock Birkh\"{a}user Boston, Inc., Boston, 2001.

\bibitem{MR2145447} {\sc M.~Bohner},
\newblock {\it {S}ome oscillation criteria for first order delay dynamic equations},
\newblock {Far East J.\ Appl.\ Math.} {\bf 18}, 3 (2005), 289--304.

\bibitem{MR2320804} {\sc M.~Bohner and G.\,Sh.~Guseinov},
\newblock {\it {T}he convolution on time scales},
\newblock {Abstr.\ Appl.\ Anal.}, Art.\,ID~58373 (2007), 24~pp.

\bibitem{MR2475963} {\sc M.~Bohner, B.~Karpuz and \"{O}.~\"{O}calan},
\newblock {\it {I}terated oscillation criteria for delay dynamic equations of first order},
\newblock {Adv.\ Difference Equ.} {\bf 2008} (2008), Art.\,ID~458687, 12~pp.

\bibitem{MR2683912} {\sc E.~Braverman and B.~Karpuz},
\newblock {\it {N}onoscillation of first-order dynamic equations with several delays},
\newblock {Adv.\ Difference Equ.} {\bf 2010} (2010), Art.\,ID~873459, 22~pp.

\bibitem{MR2596170} {\sc D.\,X.~Chen},
\newblock {\it {O}scillation and asymptotic behavior for $n$th-order nonlinear neutral delay dynamic equations on time scales},
\newblock {Acta Appl.\ Math.} {\bf 109}, 3 (2010), 703--719.

\bibitem{MR2900835} {\sc L.~Erbe, B.~Karpuz and A.~Peterson},
\newblock {\it {K}amenev-type oscillation criteria for higher-order neutral delay dynamic equations},
\newblock {Int.\ J.\ Difference Equ.} {\bf 6}, 1 (2011), 1--16.

\bibitem{MR2040224} {\sc J.~D\v{z}urina},
\newblock {\it {O}scillation theorems for neutral differential equations of higher order},
\newblock {Czechoslovak Math.\ J.} {\bf 54(129)}, 1 (2004), 107--117.

\bibitem{MR3019096} {\sc S.\,R.~Grace},
\newblock {\it {O}n the oscillation of $n$th order dynamic equations on time-scales},
\newblock {Mediterr.\ J.\ Math.} {\bf 10}, 1 (2013), 147--156.

\bibitem{MR3065777} {\sc S.\,R.~Grace, R.~Mert and A.~Zafer},
\newblock {\it {O}scillatory behavior of higher-order neutral type dynamic equations},
\newblock {Electron.\ J.\ Qual.\ Theory Differ.\ Equ.} (2013), no.~29, 15~pp.

\bibitem{MR2557101} {\sc B.~Karpuz},
\newblock {\it {A}symptotic behaviour of bounded solutions of a class of higher-order neutral dynamic equations},
\newblock {Appl.\ Math.\ Comput.} {\bf 215}, 6 (2009), 2174--2183.

\bibitem{MR2506155} {\sc B.~Karpuz},
\newblock {\it {U}nbounded oscillation of higher-order nonlinear delay dynamic equations of neutral type with oscillating coefficients},
\newblock {Electron.\ J.\ Qual.\ Theory Differ.\ Equ.} {\bf 34} (2009), 14~pp.

\bibitem{MR2793808} {\sc B.~Karpuz},
\newblock {\it {E}xistence and uniqueness of solutions to systems of delay dynamic equations on time scales},
\newblock {Int.\ J.\ Math.\ Comput.} {\bf 10}, M11 (2011), 48--58.

\bibitem{MR3066722} {\sc B.~Karpuz},
\newblock {\it {L}i type oscillation theorem for delay dynamic equations},
\newblock {Math.\ Methods Appl.\ Sci.} {\bf 36}, 9 (2013), 993--1002.

\bibitem{MR3091942} {\sc B.~Karpuz},
\newblock {\it {S}ufficient conditions for the oscillation and asymptotic behaviour of higher-order dynamic equations of neutral type},
\newblock {Appl.\ Math.\ Comput.} {\bf 221} (2013), 453--462. 

\bibitem{MR3393378} {\sc B.~Karpuz},
\newblock {\it {C}omparison tests for the asymptotic behaviour of higher-order dynamic equations of neutral type},
\newblock {Forum Math.} {\bf 27}, 5 (2015), 2759--2773. 

\bibitem{bk17} {\sc B.~Karpuz},
\newblock {\it {S}ufficient conditions for the oscillation of odd-order delay dynamic neutral equations},
\newblock {Int.\ J.\ Anal.\ Appl.} {\bf 14}, 1 (2017), 69--76. 

\bibitem{MR3518254} {\sc B.~Karpuz and \"{O}.~\"{O}calan},
\newblock {\it {N}ew oscillation tests and some refinements for first-order delay dynamic equations},
\newblock {Turkish J.\ Math.} {\bf 40}, 4 (2016), 850--863.

\bibitem{MR1635228} {\sc E.~Thandapani, S.~Sundaram, J.\,R.~Graef and P.\,W.~Spikes},
\newblock {\it {A}symptotic behaviour and oscillation of solutions of neutral delay difference equations of arbitrary order},
\newblock {Math.\ Slovaca} {\bf 47}, 5 (1997), 539--551.

\bibitem{MR2831532} {\sc T.\,X.~Sun, H.\,J.~Xi and W.\,Y.~Yu},
\newblock {\it {A}symptotic behaviors of higher order nonlinear dynamic equations on time scales},
\newblock {J.\ Appl.\ Math.\ Comput.} {\bf 37}, 1-2 (2011), 177--192.

\bibitem{MR2367669} {\sc Y.~\c{S}ahiner and I.\,P.~Stavroulakis},
\newblock {\it {O}scillations of first order delay dynamic equations},
\newblock {Dynam.\ Systems Appl.} {\bf 15}, 3-4 (2006), 645--655.

\bibitem{MR0640329} {\sc C.\,G.~Philos},
\newblock {\it {A} new criterion for the oscillatory and asymptotic behavior of delay differential equations},
\newblock {Bull.\ Acad.\ Polon.\ Sci.\ S\'{e}r.\ Sci.\ Math.} {\bf 29}, 7-8 (1981), 367--370.

\bibitem{MR3029346} {\sc D.~U\c{c}ar and Y.~Bolat},
\newblock {\it {O}scillatory behaviour of a higher-order dynamic equation},
\newblock {J.\ Inequal.\ Appl.} {\bf 52} (2013), 10~pp.

\bibitem{MR2250359} {\sc Q.\,X.~Zhang and J.~Yan},
\newblock {\it {O}scillation behavior of even order neutral differential equations with variable coefficients},
\newblock {Appl.\ Math.\ Lett.} {\bf 19}, 22 (2006), 1202--1206.

\end{thebibliography}
\end{document}